\documentclass[12pt,twoside]{amsart}
\usepackage{amsmath}
\usepackage{amsthm}
\usepackage{amsfonts}
\usepackage{amssymb}
\usepackage{latexsym}
\usepackage{mathrsfs}
\usepackage{amsmath}
\usepackage{amsthm}
\usepackage{amsfonts}
\usepackage{amssymb}
\usepackage{latexsym}
\usepackage{geometry}
\usepackage{dsfont}
\usepackage[dvips]{graphicx}
\usepackage{color}
\usepackage[all]{xy}

\date{}
\allowdisplaybreaks[4] \footskip=15pt
\renewcommand{\uppercasenonmath}[1]{}

\usepackage{graphicx,amssymb}
\usepackage[all]{xy}
\usepackage{amsmath}
\usepackage{enumerate}
\allowdisplaybreaks
\usepackage{amsthm}
\usepackage{color}
\usepackage{amsmath}
\theoremstyle{plain}
\newtheorem{theorem}{Theorem}[section]
\newtheorem{proposition}[theorem]{Proposition}

\newtheorem{corollary}[theorem]{Corollary}

\newtheorem*{open question}{Open Question}
\newtheorem{definition}[theorem]{Definition}
\newcommand{\C}{\boldsymbol{{C}}}
\newcommand{\K}{\boldsymbol{{K}}}

\newcommand{\D}{\boldsymbol{{D}}}
\theoremstyle{definition}
\newtheorem*{acknowledgement}{Acknowledgement}

\theoremstyle{remark}
\newtheorem{remark}[theorem]{Remark}

\newcommand{\Ex}{\mathcal{E}}

\newcommand{\Proj}{\mathcal{P}}
\newcommand{\PP}{\mathcal{PP}}

\newcommand{\Id}{\mathrm{Id}}

\def\next{n{\rm \mbox{-}Ext}}
\def\fPD{{\rm fPD}}

\def\Mod{{\rm Mod}}
\def\fPD{{\rm fPD}}

\def\Hom{{\rm Hom}}
\def\RHom{{\rm \textbf{R}Hom}}
\def\Ext{{\rm Ext}}

\def\Ker{{\rm Ker}}
\def\Perf{{\rm Perf}}

\def\Coker{{\rm Coker}}

\def\pd{{\rm pd}}
\def\id{{\rm id}}

\def\H{{\rm H}}

\def\fP{\mathcal{P}}
\def\fI{\mathcal{I}}
\def\Cone{{\rm Cone}}

\def\gldim{{\rm gldim}}
\def\supn{\boldsymbol{sup_n}}
\def\infn{\boldsymbol{inf_n}}

\begin{document}
	\begin{center}
		{\large  \bf On the generalizations of global dimensions and singularity categories}
		
		\vspace{0.5cm}   Xiaolei Zhang$^{a}$,~ Tiwei Zhao$^{b}$,~ Dingguo Wang$^{b}$

		{\footnotesize  a.\ School of Mathematics and Statistics, Shandong University of Technology, Zibo 255049, P.R. China\\
			b.\ School of Mathematical Sciences, Qufu Normal University, Qufu 273165, P.R. China
			
			E-mail: zxlrghj@163.com, tiweizhao@qfnu.edu.cn, dgwang@qfnu.edu.cn\\
		}
	\end{center}
	
	\bigskip
	\centerline { \bf  Abstract}
	\bigskip
	\leftskip10truemm \rightskip10truemm \noindent
	
	For each $n\in\mathbb{N}\cup\{\infty\}$, we introduce the notion of $n$-singularity category $\D_{n\mbox{-}sg}(R)$ of a given ring $R$, which can be seen as a generalization of the classical singularity category. Moreover, the $n$-global dimension $n\mbox{-}\gldim(R)$ of  $R$ is investigated. We show that  $\D_{n\mbox{-}sg}(R)=0$ if and only if $n\mbox{-}\gldim(R)$ is finite.  Furthermore, we characterize the vanishing property of  $n$-singularity categories in terms of recollements.
	\vbox to 0.3cm{}\\
	
	\bigskip

	\leftskip0truemm \rightskip0truemm
	\bigskip
	
	\section{Introduction}
	
	In this paper, $R$ always is an associative ring with identity and $R$-Mod is the category of all left $R$-modules. All $R$-modules are left $R$-modules unless otherwise stated.  Throughout this paper,  we always fix $n$ to be a non-negative integer or $\infty$. Besides, we always identify $\infty+1$  with $\infty$.
	
	For a left Noetherian ring $R$, Buchweitz introduced the notion of
	stable derived
	categories $\D_{sg}(R):=\D^{b}(R)/\Perf(R)$ the Verdier quotient of the bounded derived category of $R$ by its full subcategory of perfect complexes (see \cite{Buchweitz}). Note that $\D_{sg}(R) = 0$ if and only if $R$ has finite global dimension. So $\D_{sg}(R)$ reflects homological singularity of the ring $R$.
	Moreover, Orlov \cite{Orlov} considered the stable derived category  of all coherent sheaves over an algebraic variety $\mathcal{X}$ under the name singularity categories, since he showed that  $\D_{sg}(\mathcal{X})=0$ if and only if $\mathcal{X}$ is smooth. The singularity categories on general rings  are introduced by  Beligiannis \cite{Beligiannis}.  Let $R$ be a ring and $\D^{b}(R)$ the bounded derived category of $R$-modules. Then $\D^{b}(R)\cong \K^{-,b}(\Proj)$ the bounded  above  homotopy category of all projective modules, where $\Proj$ denotes the full subcategory of all projective modules. Beligiannis defined the singularity categories of $R$ to be
	$$\D_{sg}(R):=\D^{b}(R)/\K^{b}(\Proj)\cong \K^{-,b}(\Proj)/\K^{b}(\Proj),$$the Verdier quotient of $\D^{b}(R)$ (or $\K^{-,b}(\Proj)$)  by the bounded  homotopy category of all projective modules. Note that the Buchweitz's
	stable derived
	categories and Beligiannis's singularity category coincide on the category of all $R$-modules.

	It is well known that in the sense of Neeman \cite{Neeman}, the derived category $\D^b(R)$ is the derived category of
	the exact category $(R$-$\mathbb{\Mod}, \mathscr{E})$, where $\mathscr{E}$ is the collection of all short exact sequences
	of $R$-modules.  Zheng and Huang \cite{Zheng} introduced the pure derived category $\D^b_{pur}(R)$ as the derived category of the exact category
	$(R$-$\mathbb{\Mod}, \mathscr{E}_{pur})$, where $\mathscr{E}_{pur}$ is the collection of all  short pure-exact sequences of $R$-modules. Latter, for a ring $R$, Cao and Ren \cite{Cao} defined the pure singularity category  $\D^b_{psg}(R)=\D^b_{pur}(R)/\K^{b}(\PP)$ where $\K^{b}(\PP)$ denotes the bounded  homotopy category of all pure-projective modules. They showed $\D^b_{psg}(R)=0$ if and only if the pure-global dimension of $R$ is finite, and also characterized the vanishing properties of  pure singularity categories in terms of recollements.
	Recently, Tan, Wang and Zhao \cite{Tan} and  Zhang \cite{Zcx},  respectively, generalized the above results to $n$-pure derived category $\D^b_{n\mbox{\tiny{-pur}}}(R)$ and $n$-pure singularity categories $\D^b_{\tiny{P}_nsg}(R)$.
	
	To extend the classical derived category to a more general setting. The authors \cite{ZZW} introduced and studied the notion of $n$-derived categories $\D^b_{n}(R)$, which the derived category of
	the exact category $(R$-$\mathbb{\Mod}, \mathscr{E}_n)$, where $\mathscr{E}_n$ is the collection of all short $n$-exact sequences
	of $R$-modules. Note that the $0$-derived category $\D^b_{0}(R)$ is precisely the classical derived category $\D^b(R)$. The main motivation of this paper is to generalize the classical singularity categories, and investigate the $n$-singularity categories $\D^b_{n\mbox{-}sg}(R)$ which is the Verdier quotient of  $\D^b_{n}(R)$ by $\K^b(\fP_n)$ the bounded homotopy category of all $n$-projective $R$-modules. The paper is organized as follows. In Section 2, we give a quick review on $n$-projective ($n$-injective) modules, $n$-exact sequences and $n$-derived categories.  In Section 3, we investigate the relative homological dimensions, that is,  $n$-projective dimensions  and $n$-injective dimensions of complexes, and the $n$-global dimension $n\mbox{-}\gldim(R)$ of a given ring $R$. In Section 4, we introduce and study the $n$-singularity categories. We obtain that for a ring $R$, $\D^b_{n\mbox{-}sg}(R) = 0$ if and only if $n\mbox{-}\gldim(R)$ is finite. Moreover, we show that for any rings $R$, $S$ and $T,$ if
	$\D^b_{n}(R)$ admits a recollement${\xymatrix{\D^b_{n}(S)\tiny{\ar[rr]!R|{}}&&{\ar@<-1ex>[ll]!R|{}\ar@<1ex>[ll]!R|{}}\D^b_{n}(R){\ar[rr]!L|{}}&&{\ar@<-1ex>[ll]!L|{}\ar@<1ex>[ll]!L|{}}\D^b_{n}(T)}},$ then $\D^b_{n\mbox{-}sg}(R) = 0$ if and only if $\D^b_{n\mbox{-}sg}(S) = 0 = \D^b_{n\mbox{-}sg}(T)$.


	\section{Preliminaries}
	We give a quick review on $n$-projective  modules, $n$-injective modules, $n$-exact sequences and $n$-derived categories. For more details, refer to \cite{ZZW}.
	
	Let $R$ be a ring and  $M$  a left $R$-module. We say $M$ has a \emph{finite projective resolution} $($of length at most $n)$, if there is a finite exact sequence
	$$0\rightarrow P_n\rightarrow P_{n-1}\rightarrow \dots\rightarrow P_1\rightarrow P_0\rightarrow M\rightarrow 0$$ where each $P_i$ is a finitely generated projective left $R$-module. The class of all $R$-modules with finite projective resolutions $($resp., of length less than $n)$ is denoted by $\Proj^{< \infty}$ (resp., $\Proj^{< n}$).
	
	Recall from \cite{ZZW}  that a short exact sequence of left $R$-modules $$0\rightarrow A\rightarrow B\rightarrow C\rightarrow 0$$ is said to be \emph{$n$-exact} provided that $$0\rightarrow \Hom_R(M,A)\rightarrow \Hom_R(M,B)\rightarrow \Hom_R(M,C)\rightarrow 0$$ is exact for any $M\in \Proj^{<n+1}$.
	If we denote by  $\mathscr{E}_{n}$ to be the class of all $n$-exact sequences, then $\mathscr{E}_{n}$ is an  exact structure of $R$-Mod in the sense of Quillen \cite{Q73}. If we denote $\mathscr{E}_{fp}$ to be the class of all classical pure exact sequences, then we have the following inclusions:
	$$\mathscr{E}_{fp}\subseteq \mathscr{E}_{\infty}\subseteq\cdots\subseteq\mathscr{E}_{{n+1}}\subseteq \mathscr{E}_{n} \subseteq\cdots\subseteq \mathscr{E}_{{1}}\subseteq \mathscr{E}_{{0}}.$$

	Recall  from \cite{ZZW}  that an $R$-module $M$ is said to be \emph{$n$-projective}
	if $M$ is projective with respect to all $n$-exact sequences, i.e.,  for any $n$-exact sequence $ 0\rightarrow A\rightarrow B\rightarrow C\rightarrow 0$, the natural exact sequence $$0\rightarrow \Hom_R(M, A)\rightarrow \Hom_R(M,B)\rightarrow \Hom_R(M,C)\rightarrow 0$$ is exact. The notion of  \emph{$n$-injective modules} can be defined dually.
	
	We always denote by $\fP_n$ (resp., $\fI_n$) class of $n$-projective (resp., $n$-injective) $R$-modules. Trivially,
	
	$$\fP_0\subseteq  \fP_1\subseteq \cdots\subseteq \fP_n\subseteq \cdots\subseteq \fP_\infty,$$
	$$\fI_0\subseteq  \fI_1\subseteq \cdots\subseteq \fI_n\subseteq \cdots\subseteq \fI_\infty.$$

	\begin{remark}\label{str-fpr} The following statements hold.
		\begin{enumerate}
			\item  $0$-projective modules are exactly  projective modules and $0$-injective modules are exactly  injective modules. 
			\item  If $n\geq 1$, then the class of $n$-projective (resp., $n$-injective) modules is not closed under extensions in general.  Indeed, let $R$ be the Kronecker algebra. Then $R$ is an Artin ring with global dimension equal to $1$. So the class of  $n$-projective (resp., $n$-injective) modules is exactly that of pure projective (resp., pure injective) modules. However, the class of pure projective (resp., pure injective) modules is not  closed under extensions in general (see \cite[Example 8.1.18]{Prest2}, \cite{O92} and \cite{O02} for more details). 
			\item 	If $n\geq 1$, then the class of $n$-projective (resp., $n$-injective) modules is not closed under the kernel of epimorphisms (the cokernel of monomorphisms) in general. Indeed, let $R$ be a  finite-dimensional hereditary algebra of infinite represented type.  So pure projective (resp., pure injective) modules are precisely $n$-projective (resp., $n$-injective). 
			Note that the hereditary algebra  $R$ is not pure semisimple. So there are $R$-modules which are not pure-projective and $R$-modules which are not pure-injective.  However, each projective $R$-module (resp., injective $R$-module), being a direct sum of finite-dimensional modules, is  pure-injective (resp., pure-projective).  So the minimal injective coresolution of non-pure-projective modules (resp., minimal projective resolution of non-pure-injective modules) gives the counterexamples.
		\end{enumerate}
	\end{remark}

	It is well-known that every $R$-module has a pure-injective envelope and pure-projective precover. Moreover, it follows by \cite[Corollary 2.3]{S21} that every $R$-module has a pure   projective cover if and only if every direct limit of pure   projective $R$-modules is pure   projective, if and only if $R$ is a pure semisimple ring (i.e., every $R$-modules is pure   projective).
	Recall from \cite{ZZW} that  every $R$-module has an $n$-envelope and a $n$-precover. Moreover, every $R$-module has a 
	$n$-projective cover if and only if every direct limit of 
	$n$-projective $R$-modules is 
	$n$-projective, if and only if the $R$-module $\bigoplus\limits_{P\in [\Proj^{<n+1}]}P$  has  perfect decompositions, where $[\Proj^{<n+1}]$ denotes the representatives of  isomorphism classes of modules in $\Proj^{<n+1}$.
	\begin{remark}\label{ppr-M}
		It follows by the above  that every $R$-module admits  $n$-projective resolutions and  $n$-injective coresolutions, that is, for each $M \in R$-\Mod,
		there exists an exact sequence
		$$\cdots \rightarrow P_1 \rightarrow P_0 \rightarrow M \rightarrow 0$$
		with each $P_i\in \fP_n $ and it remains exact after applying $\Hom_R(P, -)$ for any $P\in\fP_n,$
		and there exists an exact sequence
		$$0\rightarrow M\rightarrow E_0\rightarrow E_1 \rightarrow\cdots $$
		with each $E_i \in \fI_n $ and it remains exact after applying $\Hom_R(-, E)$ for any $E\in \fI_n.$
	\end{remark}

	We denote by $\C(R)$ and $\K(R)$ the category of complexes of $R$-Mod and homotopy category of $R$-Mod, respectively. For any $X^{\bullet}\in \C(R)$, we write
	$$X^{\bullet}:=\cdots\rightarrow X^{i-1}\xrightarrow{d_{X^{\bullet}}^{i-1}}X^{i}\xrightarrow{d_{X^{\bullet}}^{i}} X^{i+1}\xrightarrow{d_{X^{\bullet}}^{i+1}} X^{i+2}\rightarrow\cdots$$
	Let $X^{\bullet}\in \C(R)$. If $X^{i}=0$ for $i\gg 0$, then $X^{\bullet}$ is said to be bounded above; if $X^{i}=0$ for $i\ll 0$, then $X^{\bullet}$ is said to be bounded below;  if it is both bounded above and  bounded below, then $X^{\bullet}$ is said to be bounded.

	Recall from \cite{ZZW} that an exact complex $X^{\bullet}$ is said to be $n$-exact at $i$  if  the sequence $0\rightarrow \Ker(d_{X^{\bullet}}^{i})\rightarrow X^{i}\rightarrow \Coker(d_{X^{\bullet}}^{i-1})\rightarrow 0$ is $n$-exact.  A complex $X^{\bullet}$ is said to be $n$-exact if it is $n$-exact at $i$ for any   $i\in \mathbb{Z}$. The class of all $n$-exact sequences is denoted by $\Ex_n$. Let $f:X^{\bullet}\rightarrow Y^{\bullet}$ be a cochain map of complexes. We say $f$ is an  $n$-  quasi-isomorphism provided that its mapping cone $\Cone^{\bullet}(f)$ is an $n$-exact complex.

	The left little finitistic dimension of a ring $R$, denoted by $l.\fPD(R)$, is defined to be the supremum of the projective dimensions of all left $R$-modules in $\Proj^{<\infty}$.
	Clearly, $l.\fPD(R)\leq n$ if and only if $\Proj^{<\infty}=\Proj^{< n+1}$. Using the notations as above, we have the following characterizations of little finitistic dimension $l.\fPD(R)$ of a given ring $R$.
	
	\begin{theorem} \cite{ZZW} \label{fpd-exact-seq} Let $R$ be a ring. Then the following statements are equivalent.
		\begin{enumerate}
			\item $l.\fPD(R)\leq n$.
			\item $\mathscr{E}_{n}=\mathscr{E}_{{n+1}}\ ($or $ \fP_n= \fP_{n+1}$,  or $\Ex_n=\Ex_{n+1})$.
			\item $\mathscr{E}_{n}=\mathscr{E}_{{m}} ($or $ \fP_n= \fP_{m}$, or $\Ex_n=\Ex_{m})$ for some $m>n$.
			\item $\mathscr{E}_{n}=\mathscr{E}_{{m}} ($or $ \fP_n= \fP_{m}$,  or $\Ex_n=\Ex_{m})$ for any $m>n$.
			\item $\mathscr{E}_{n}=\mathscr{E}_{\infty} ($or $ \fP_n= \fP_{\infty}$,  or $\Ex_n=\Ex_{\infty})$.
		\end{enumerate}
	If $R$ is commutative, all above are equivalent to 
	
	$\textit{(6)}$ $ \fI_n= \fI_{n+1},$ ( or $ \fI_n= \fI_{m}$ for some $m>n$, or $ \fI_n= \fI_{m}$  for any $m>n$, or $ \fI_n= \fI_{\infty}$.)

	\end{theorem}
	Recall the notion of $n$-derived category of a given ring $R$  from \cite{ZZW}. 
	For any $\ast\in \{\mbox{blank},+,-,b\}$,
	set
	\begin{center}
		$\K^{\ast}_{n}(R)=\{X^{\bullet}\in \K^{\ast}(R)\mid  X^{\bullet}$  is an $n$-exact complex$\}$.
	\end{center}
	Then
	$\K^{\ast}_{n}(R)$ is a  thick subcategory of $\K^{\ast}(R)$.   The Verdier quotient
	$$\D^{\ast}_{n}(R):=\K(R)/\K_{n}(R)$$
	is said to be the \emph{$n$-derived category}  of $R$. Trivially, $0$-derived category is precisely the classical derived category.

	Note that $\K^{\ast}_{m}(R)$ is also a  thick subcategory of $\K^{\ast}_{n}(R)$ for any $m>n$. Set $\K^{\ast}_{n,m}(R)=\K^{\ast}_{n}(R)/\K^{\ast}_{m}(R)$ the  Verdier quotient of $\K^{\ast}_{n}(R)$ by $\K^{\ast}_{m}(R)$. Then $\K^{\ast}_{n,m}(R)$ is a  thick subcategory of $\D^{\ast}_{n}(R)$ for any $m>n$. The left little finitistic dimensions can also be characterized in terms of $n$-derived categories.
	
	\begin{theorem} \cite{ZZW} \label{fpd-derived} Let $R$ be a ring. Then for any $\ast\in \{\mbox{blank},+,-,b\}$, there is a  triangulated equivalence for each $n$:
		$$\D^{\ast}_{n}(R)\cong \D^{\ast}_{m}(R)/\K^{\ast}_{n,m}(R).$$
		
		Moreover, the following statements are equivalents:
		\begin{enumerate}
			\item $l.\fPD(R)\leq n$.
			\item $\K^{\ast}_{n,m}(R)=0$ for any $m>n$.
			\item $\K^{\ast}_{n,m}(R)=0$ for some $m>n$.
			\item $\D^{\ast}_{m}(R)$ is naturally isomorphic to $\D^{\ast}_{n}(R)$ for any $m>n$.
			\item $\D^{\ast}_{m}(R)$ is naturally isomorphic to $\D^{\ast}_{n}(R)$ for some $m>n$.
		\end{enumerate}
	\end{theorem}

	\section{the related homological dimensions}
	
	We begin with the notions of $n$-projective  dimensions  and $n$-injective dimensions of complexes in $ \D^b_{n}(R)$. 
	\begin{definition}
		
		Let $X^{\bullet}\in \D^b_{n}(R)$.
		
		\begin{enumerate}
			\item An $n$-projective resolution of $X^{\bullet}$ is an $n$-quasi-isomorphism $f:P^{\bullet}\rightarrow X^{\bullet}$ with $P^{\bullet}$ a complex of $n$-projective $R$-modules satisfying $\Hom_R(P^{\bullet},-)$ preserves $n$-exact complexes. Dually, an $n$-injective coresolution of $X^{\bullet}$ is defined.
			\item $X^{\bullet}$ is said to have $n$-projective dimension at most $m$, written $n$-\pd$_RX^{\bullet}\leq m$ if there exists an $n$-projective resolution $P^{\bullet}\rightarrow X^{\bullet}$ with $P^{i}=0$ for any $i<-m$. If  $n$-\pd$_RX^{\bullet}\leq m$  for all $m$, then we write $n$-\pd$_RX^{\bullet}=-\infty$; and if there exists no $m$ such that $n$-\pd$_RX^{\bullet}\leq m$ , then we write $n$-\pd$_RX^{\bullet}=\infty$.
			\item the $n$-injective dimension $n$-\id$_RX^{\bullet}$ can be defined dually.
		\end{enumerate}
	\end{definition}

	\begin{remark}
		Let $X^{\bullet}\in \D^b_{n}(R)$.
		\begin{enumerate}
			\item The  $n$-projective (resp., $n$-injective) resolution of a stalk complex is precisely  that of a module defined in Remark  \ref{ppr-M}.
			
			\item Let $X^{\bullet}\in \D^b_{n}(R)$. Then $n$-\pd$_RX^{\bullet}=-\infty$ (equivalently, $n$-\id$_RX^{\bullet}=\infty$)  is equivlent to that $X^{\bullet}$ is an $n$-exact complex.
			\item  $n$-\pd$_RX^{\bullet}=-\sup\{\inf\{m \in\mathbb{Z}\mid P_m = 0\} \mid P^{\bullet}\rightarrow X^{\bullet}$ is an $n$-projective resolution$\}$.
			\item  $n$-\id$_RX^{\bullet}=\inf\{\sup\{m \in\mathbb{Z}\mid I_m = 0\} \mid X^{\bullet}\rightarrow I^{\bullet}$ is an $n$-injective resolution$\}$.
		\end{enumerate}
	\end{remark}
	
	\begin{proposition}\label{pd-fpd}  Let $M$ be an $R$-module. 
		\begin{enumerate}
			\item If $n$-\pd$_RM\leq m$. Then $m\leq \pd_RM\leq n+m$. 
			\item If $n$-\id$_RM\leq m$. Then $m\leq \id_RM\leq n+m$.
		\end{enumerate}
	\end{proposition}
	\begin{proof} (1) If $n=\infty$, the right inequality certainly holds. So we can assume $n<\infty$. Suppose $n$-\pd$_RM\leq m$. Then there exists an exact sequence
		$$0 \rightarrow P_m\xrightarrow{d_m}P_{m-1}\xrightarrow{d_{m-1}} \cdots \rightarrow P_1 \xrightarrow{d_1} P_0 \xrightarrow{d_0} M \rightarrow 0$$
		with each $P_i\in \fP_n.$ Note each $P_i$ has projective dimension at most $n$. Setting $C_{m-1}:=\Coker(d_m)$,  there is a short exact sequence $0\rightarrow P_m\rightarrow P_{m-1}\rightarrow C_{m-1}\rightarrow0.$ Then for any $R$-module $N$, we have an exact sequence $$0=\Ext_R^{n+1}(P_m,N)\rightarrow\Ext_R^{n+2}(C_{m-1},N)\rightarrow\Ext_R^{n+2}(P_{m-1},N)=0.$$ Hence $\Ext_R^{n+2}(C_{m-1},N)$ for any $R$-module $N$, and so $\pd_RC_{m-1}\leq n+1.$ Iterating these steps, we have \pd$_RM\leq n+m$.  Since each projective module is $n$-projective, we have $\pd_RM\geq  n$-\pd$_RM=m$.
		
		(2) one can prove it dually.
	\end{proof}
	
	\begin{remark} The converse of Proposition \ref{pd-fpd} is not true. Indeed, let $\mathbb{Z}$ be ring of integers with $\mathbb{Q}$ its quotient field.  Then $\pd_{\mathbb{Z}}\mathbb{Q}=1$. However, by \cite[Remark 3.5]{ZZW}, $\mathbb{Q}$ is not $1$-projective. So  $1$-\pd$_{\mathbb{Z}}\mathbb{Q}=1<1+1=2$.  Similarly, $\id_{\mathbb{Z}}\mathbb{Z}=1$. Since $\mathbb{Z}$ is not $1$-injective, we also have  $1$-\pd$_{\mathbb{Z}}\mathbb{Z}=1<2$. 
	\end{remark}

	Let $X^{\bullet}\in \D^b_{n}(R)$. Then by \cite[Proposition 5.10]{ZZW}, there exists a complex $P^{\bullet}\in \K^{-}(\fP_n)$ such
	that $P^{\bullet}\cong X^{\bullet}$  in $\D^b_{n}(R)$. By \cite[Lemma 4.8]{ZZW}, the functor $\Hom_R(P^{\bullet},-)$ preserves $n$-exact complexes, and hence preserves $n$-quasi-isomorphisms. Then we get an $n$-quasi-isomorphism from $P^{\bullet}$
	to $X^{\bullet}$. The statements for the $n$-injective version are dual. Following above, we have the following result.

	\begin{corollary}
		Let $X^{\bullet}\in \D^b_{n}(R)$. Then
		$X^{\bullet}$ admits an $n$-projective $($resp., $n$-injective$)$ resolution $($resp., coresolution$).$
	\end{corollary}
	Now we can define a functor
	$$\RHom_R(-, -): \D^b_{n}(R)^{op}\times \D^b_{n}(R) \rightarrow \D_{n}(\mathbb{Z})$$
	using either the $n$-projective resolution of the first variable or the $n$-injective coresolution of the second variable. More precisely, let $P_{X^{\bullet}}$ be an $n$-projective resolution of $X^{\bullet}$ and
	$I_{Y^{\bullet}}$ an $n$-injective coresolution of $Y^{\bullet}$ . Then we have the following $n$-quasi-isomorphisms
	$$\RHom_R(X^{\bullet}, Y^{\bullet}) \longrightarrow \Hom_R(P_{X^{\bullet}}, I_{Y^{\bullet}} )\longleftarrow  \RHom_R(X^{\bullet}, Y^{\bullet}),$$
	where $\RHom_R(X^{\bullet}, Y^{\bullet}):= \Hom_R(P_{X^{\bullet}}, Y^{\bullet})$ and $\Hom_R(X^{\bullet}, I_{Y^{\bullet}} ) := \RHom_R(X^{\bullet}, Y^{\bullet}).$
	It follows that $\RHom_R(-,-)$ is well defined, and we call it the right $n$-derived functor of
	Hom. In order to coincide with the classical ones, we put
	$$\next^i_R(X^{\bullet}, Y^{\bullet} ) := \H^i\RHom_R(X^{\bullet}, Y^{\bullet}) = \H^i\Hom_R(P_{X^{\bullet}}, Y^{\bullet}) = \H^i\Hom_R(X^{\bullet}, I_{Y^{\bullet}} ).$$
	For any $X^{\bullet}\in \C(R),$ we write
	$$\infn X^{\bullet} := \inf\{m \in\mathbb{Z} \mid X^{\bullet}\ \mbox{is not}\ n\mbox{-exact at}\  m\},$$
	$$\supn X^{\bullet} := \sup\{m \in\mathbb{Z} \mid X^{\bullet}\ \mbox{is not}\ n\mbox{-exact at}\  m\}.$$
	If $X^{\bullet}$ is not $n$-exact at $m$ for any $m$, then we set $\infn X^{\bullet}=-\infty$ and $\supn X^{\bullet}=\infty$. If $X^{\bullet}$
	is $n$-exact at $m$ for all $m,$ that is, $X^{\bullet}$ is an $n$-exact complex, then we set $\infn X^{\bullet}=\infty$
	and $\supn X^{\bullet}=-\infty$.
	Recall that $X^{\bullet}\in \C(R) $ is called contractible if it is isomorphic to the zero object in $\K(R),$
	equivalently, the identical map $\Id_{X^{\bullet}}$ is homotopic to zero, that is, $X^{\bullet}$ is splitting exact by \cite[Exercise 1.4.3]{Weibel}. The following result gives some criteria for computing $n$-projective
	dimension in terms of the $n$-projective resolutions and $n$-derived functors.

	\begin{theorem}\label{c-fpn}
		Let $X^{\bullet} \in\D^b_{n}(R)$ and $m \in \mathbb{Z}.$ Then the following statements are equivalent.
		\begin{enumerate}
			\item $n$-\pd$_RX^{\bullet} \leq m.$
			\item  $\infn X^{\bullet}\geq -m$, and if $f: (P')^{\bullet}\rightarrow X^{\bullet}$ is an $n$-projective resolution of $X^{\bullet},$ then the $R$-module $\Coker(d_{(P')^{\bullet}}^{-m-1})$  is $n$-projective.
			\item  If $f: (P')^{\bullet}\rightarrow X^{\bullet}$ is an $n$-projective resolution of $X^{\bullet}$, then $(P')^{\bullet}= (P_1)^{\bullet}\oplus (P_2)^{\bullet},$ where $(P')_1^i = 0$  for any $i <-m$ and $(P_2)^{\bullet}$ is contractible.
			\item  $\next^i_R(X^{\bullet}, Y^{\bullet}) = 0$ for any $Y^{\bullet} \in \D^b_{n}(R)$ and $i > m + \supn Y^{\bullet}$.
			\item  $\infn X^{\bullet}\geq -m$ and $\next_R^{m+1}(X^{\bullet}, N) = 0$ for any $N \in $R-\Mod.
		\end{enumerate}
	\end{theorem}
	\begin{proof} $(1)\Rightarrow (2)$ Suppose $n$-\pd$_RX^{\bullet} \leq m.$ Then there exists an $n$-projective resoltuion $f:P^{\bullet}\rightarrow X^{\bullet}$ with $P^i=0$ for any $i<-m$. It follows by  \cite[Proposition 5.9]{ZZW} that $\infn X^{\bullet}\geq -m$. Let $f':(P')^{\bullet}\rightarrow X^{\bullet}$ be another $n$-projective resolution of $X^{\bullet}$. Then there exists a quasi-isomorphism of complexes
		$$\Hom_R(P^{\bullet},f'):\Hom_R(P^{\bullet},(P')^{\bullet})\rightarrow\Hom_R((P)^{\bullet},(X)^{\bullet}).$$So there exists a cochain map $g:P^{\bullet}\rightarrow (P')^{\bullet}$ such that $f'\circ g=f$. Therefore, 
		$$\Hom_R(F,f')\circ \Hom_R(F,g)=\Hom_R(F,f)$$
		for any $F\in\fP_n$. And so $g$ is an $n$-quasi-isomorphism by \cite[Proposition 4.6]{ZZW}. Then $g$ is a homotopy equivalence by \cite[Proposition 4.9]{ZZW}. It is easy to check that the exact sequence
		$$\cdots\rightarrow (P')^{-m-1}\xrightarrow{d_{(P')^{\bullet}}^{-m-1}}
		(P')^{-m}\xrightarrow{d_{(P')^{\bullet}}^{-m}}
		\Coker(d_{(P')^{\bullet}}^{-m})\rightarrow 0$$ is contractible, and hence $\Coker(d_{(P')^{\bullet}}^{-m})$ is $n$-projective.
		
		$(2)\Rightarrow (3)$ Let $f':(P')^{\bullet}\rightarrow X^{\bullet}$ be an $n$-projective resolution of $X^{\bullet}$. Then $\infn (P')^{\bullet}=\infn X^{\bullet}\geq-m$. So we have the sequence 	
		$$\cdots\rightarrow(P')^{-m-1}\xrightarrow{d_{(P')^{\bullet}}^{-m-1}}(P')^{-m}\xrightarrow{d_{(P')^{\bullet}}^{-m}}\Coker(d_{(P')^{\bullet}}^{-m})\rightarrow 0$$
		is $n$-exact. Because  $\Coker(d_{(P')^{\bullet}}^{-m})$ is $n$-projective by assumption, the above sequence is contractible. Set $(P')^{-m}=M\oplus \Coker(d_{(P')^{\bullet}}^{-m})$, and put $$(P_1)^{\bullet}:=\cdots\rightarrow0\rightarrow\Coker(d_{P'}^{-m})\rightarrow (P')^{-m+1}\rightarrow (P')^{-m+2}\rightarrow\cdots,\mbox{and}$$
		$$(P_2)^{\bullet}:=\cdots\rightarrow(P')^{-m-2}\rightarrow(P')^{-m-1}\rightarrow M\rightarrow 0\rightarrow \cdots.$$
		Then we have $(P')^{\bullet}=(P_1)^{\bullet}\oplus (P_2)^{\bullet}$, where $P_1^i=0$ for any $i<-m$ and $(P_2)^{\bullet}$ is contractible.	
		
		$(3)\Rightarrow (1)$ It follows by $(3)$ that the embedding
		$(P_1)^{\bullet}\hookrightarrow (P)^{\bullet}$ is  an $n$-quasi-isomorphism. This implies that $(X)^{\bullet}$ admits an $n$-projective resolution $(P_1)^{\bullet}\hookrightarrow (P)^{\bullet} \rightarrow (X)^{\bullet}$
		with $P^i_1 = 0$ for any $i<-m.$
		
		$(3)\Rightarrow (4)$ We only need to consider the case $\supn Y^{\bullet} = t <\infty$. Let $P^{\bullet}\rightarrow X^{\bullet}$
		be an $n$-projective resolution of $X^{\bullet}$. Then $P^{\bullet}= (P_1)^{\bullet} \oplus (P_2)^{\bullet}$, where $P_1^i = 0$ for any $i < -m$ and
		$(P_2)^{\bullet}$ is contractible. So we have
		$$\next^i_R(X^{\bullet}, Y^{\bullet}) = \H^i\Hom_R(P^{\bullet}, Y^{\bullet}) = \H^i\Hom_R((P_1)^{\bullet},Y ^{\bullet}).$$
		It follows from \cite[Proposition 5.6]{ZZW} that if  $Y^{\bullet}$ be the right canonical truncation complex of $Y^{\bullet}$ at degree
		$t$, then the embedding $(Y')^{\bullet}\hookrightarrow (Y)^{\bullet}$ is an $n$-quasi-isomorphism. So we have
		$$\H^i\Hom_R((P_1)^{\bullet},(Y)^{\bullet})=\H^i\Hom_R((P_1)^{\bullet},(Y')^{\bullet})=0$$
		for any $i > m + t$. Thus $\next^i_R(X^{\bullet},Y^{\bullet}) = 0$ for any $Y^{\bullet}\in\D_{n}(R)$ and $i>m +\supn Y^{\bullet}$.
		
		$(4)\Rightarrow (5)$ For any $N\in R$-\Mod, we have $\supn N = 0$ and $\next^{m+1}_R(X^{\bullet}, N) = 0$ by (4).
		Let $M$ be a right $R$-module. Then
		$$\H^i((M\otimes_R X^{\bullet})^{+}) = \H^i(\Hom_R(X^{\bullet}, M^{+})) = \next^i_ R(X^{\bullet}, M^{+}) = 0$$
		for any $i > m$ by the adjoint isomorphism theorem and (4). So $M\otimes_R X^{\bullet}$ is exact in degree $< -m,$
		and hence $X^{\bullet}$ is $n$-exact in degree $< -m$ by \cite[Proposition 5.9]{ZZW}. It implies that $\infn X^{\bullet} \geq -m.$
		
		$(5)\Rightarrow(3)$ Let $P^{\bullet}$ be an $n$-projective
		resolution of $X^{\bullet}$ and $N \in R$-\Mod. Then we have
		$\infn P^{\bullet}=\infn X^{\bullet}\geq-m.$
		So $P^{\bullet}$ is $n$-exact in degree $\leq-m-1$, and hence the sequence
		$$\cdots\rightarrow (P')^{-m-2}\rightarrow (P')^{-m-1}\xrightarrow{d^{-m-1}_{(P')^{\bullet}}} (P')^{-m}\rightarrow \Coker(d^{-m-1}_{(P')^{\bullet}}) \rightarrow 0$$
		is $n$-exact and it is an $n$-projective resolution of $\Coker(d^{-m-1}_{(P')^{\bullet}})$. We have the following
		equalities
		$$\next^1_R(\Coker(d^{-m-1}_{(P')^{\bullet}}),N)
		=\H^{m+1}\Hom_R(P^{\bullet}, N) = \next^{m+1}_R(X^{\bullet}, N) = 0.$$
		It implies that $\Coker(d^{-m-1}_{(P')^{\bullet}})$ is $n$-projective. Thus the above  $n$-exact complex is contractible, and therefore $P^{\bullet}=(P_1)^{\bullet}\oplus (P_2)^{\bullet}$, where $P^1_i=0$ for any $i<-m.$
	\end{proof}
	
	As for $n$-injective dimension $n$-\pd$_RX^{\bullet}$ of  $X^{\bullet} \in\D^b_{n}(R)$, one can prove the the following result dually.
	\begin{theorem}\label{c-fin}
		Let $X^{\bullet} \in\D^b_{n}(R)$ and $m \in \mathbb{Z}.$ Then the following statements are equivalent.
		\begin{enumerate}
			\item $n$-\id$_RX^{\bullet} \leq m.$
			\item  $\supn X^{\bullet}\leq m$, and if $f': X^{\bullet}\rightarrow (I')^{\bullet}$ is an $n$-injective resolution of $X^{\bullet},$ then the $R$-module $\Coker(d^m_{(I')^{\bullet}})$
			is $n$-injective.
			\item  If $f': X^{\bullet}\rightarrow I^{\bullet}$ is an $n$-injective resolution of $X^{\bullet}$, then $(I')^{\bullet}= (I_1)^{\bullet}\oplus (I_2)^{\bullet},$ where $I_1^i = 0$  for any $i >m$ and $(I_2)^{\bullet}$ is contractible.
			\item  $\next^i_R(Y^{\bullet},X^{\bullet}) = 0$ for any $Y^{\bullet} \in \D_{n}(R)$ and $i > m -\infn Y$.
			\item  $\supn X^{\bullet}\leq m$ and $\next_R^{m+1}(N, X^{\bullet}) = 0$ for any $N \in $R-\Mod.
		\end{enumerate}
	\end{theorem}

	We can characterize $n$-projective dimension and $n$-injective dimension of an complex in $D^b_{n}(R)$ in terms of the vanishing of $n$-derived functors.
	\begin{corollary}
		Let $X^{\bullet}\in D^b_{n}(R).$ Then
		\begin{center}
			$n$-$\pd_RX^{\bullet}=\sup\{i \in\mathbb{Z}\mid \next^i_ R(X^{\bullet}, N) = 0$ for some $N\in R$-$\Mod\}$,
		\end{center}
		\begin{center}
			$n$-$\id_RX^{\bullet} = \sup\{i\in\mathbb{Z}\mid \next^i_ R(M,X^{\bullet}) = 0$ for some $M\in R$-$\Mod\}$.
		\end{center}	
	\end{corollary}
	\begin{proof}
		They follow from Theorem \ref{c-fpn} and  Theorem \ref{c-fin}, respectively.
	\end{proof}

	\begin{remark}
		
		Let $N$ be an $R$-module and $m$ a nonnegative integer. Then $n$-$\id_RN\leq m$ if and only if $N$ has
		an $n$-injective coresolution of length $m$. Note that $n$-$\id_RN$ can also be defined to
		be the smallest positive integer $m$ such that $\next^{m+1}_R (M,N)=0$ for any module $M$, and set
		$n$-$\id_RN=\infty$ if no such $m$ exists. It is standard to show that these two definitions of $n$-injective dimension coincide with the definition of complexes. The $n$-projective dimension $n$-$\id_RM$ of an $R$-module can be clarified similarly.  
	\end{remark}
	
	The following result shows that the relative derived functors of $\Hom$ with respect to 
	$n$-projective modules can be interpreted as the morphisms in the corresponding relative derived
	category with respect to $n$-projective modules.
	
	\begin{proposition}\label{n-ext} Let $M$, $N$ be two $R$-modules. Then
		$$\next_{R}^{i}(M,N)\cong \Hom_{D^{b}_{n}(R)}(M,N[i]).$$
	\end{proposition}
	\begin{proof}
		Let $P^{\bullet} \rightarrow M \rightarrow 0$ be an $n$-projective resolution of $M$. Then $P^{\bullet}\rightarrow M$ is an $n$-quasi-isomorphism, and $P^{\bullet}\cong M$ in $D^b_{n}(R)$. It follows by \cite[Proposition 5.5, Proposition 5.7]{ZZW} that  the following isomorphisms of abelian groups hold:
		\begin{align*}
			\next^i_R(M, N)&\cong \H^i \Hom_R(P^{\bullet},N) \\
			&\cong\Hom_{\K(R)}(P^{\bullet},N[i])\\
			&\cong\Hom_{\D_{n}(R)}(P^{\bullet}, N[i])\\
			&\cong\Hom_{\D^b_{n}(R)}(M,N[i]).
		\end{align*}
	\end{proof}
	
	We introduce the notion of the  $n$-global dimension of a given $R$.
	\begin{definition} The left $n$-global dimension of a ring $R$, denoted by $n\mbox{-}\gldim(R)$, is defined as
		\begin{align*}
			n\mbox{-}\gldim(R)&= \sup\{n\mbox{-}\pd_RM\mid M \in R\mbox{-}\Mod\} \\
			&=\sup\{n\mbox{-}\id_RM\mid M \in R\mbox{-}\Mod\}.
		\end{align*}
	\end{definition}
	
	Trivially, the left $0$-global dimension $0\mbox{-}\gldim(R)$ of a ring $R$ is exactly the global dimension $\gldim(R)$ of $R$. The following results are easily to obtain.
	\begin{proposition}\label{inqfn} Let $R$ be a ring. Then the following statements hold.
		\begin{enumerate}
			\item If $m\geq n$, then 	$n\mbox{-}\gldim(R)\geq m\mbox{-}\gldim(R)$.
			\item If $l.\fPD(R)\leq n$,  then $n\mbox{-}\gldim(R)= m\mbox{-}\gldim(R)$ for any $m\geq n$.
			\item Suppose $n\mbox{-}\gldim(R)=k$. Then $k\leq \gldim(R)\leq n+k$.
		\end{enumerate}
	\end{proposition}
	\begin{proof}
		(1) is trivial, (2) follows by Theorem \ref{fpd-exact-seq}, and (3) follows by Proposition \ref{pd-fpd}.
	\end{proof}
	
	\begin{remark} We consider some examples on the inequality of Proposition \ref{inqfn}(3).	
		\begin{enumerate}
			\item  Let $R$ be a finite-dimensional hereditary algebra of finite-representation type. Then every $R$-module is $1$-projective, and so $1\mbox{-}\gldim(R)=0$. Hence  $0=1\mbox{-}\gldim(R)<\gldim(R)=1+1\mbox{-}\gldim(R)=1$.
			\item  Let $R$ be a finite-dimensional hereditary algebra of infinite-representation type. Then there exists an  $R$-module which is not $1$-projective, and so $1\mbox{-}\gldim(R)=1$. Hence  $1=1\mbox{-}\gldim(R)=\gldim(R)<1+1\mbox{-}\gldim(R)=2$.
		\end{enumerate}
	\end{remark}
	

	Moreover, we have the following characterizations of  $n$-global dimensions.
	\begin{theorem} Let $R$ be a ring with $n\mbox{-}\gldim(R)<\infty$. Then the following statement holds.
		\begin{center}
			$n\mbox{-}\gldim(R)=\sup\{n$-$\id_RP\mid P\in\fP_n\}=\sup\{n$-$\pd_RI\mid I\in \fI_n\}$.
		\end{center}
	\end{theorem}
	
	\begin{proof}
		We will only prove the first equality since the other one is similar. Denote by $\sup\{n$-$\id_RP\mid P\in\fP_n\}=d$ and  $n\mbox{-}\gldim(R)=m$. Trivially, $m\geq d$. We can assume that $d$ is finite. It is
		sufficient to prove that $m \leq d$ whenever $m$ is finite. Let
		$$0\rightarrow P_m \xrightarrow{d_m} P_{m-1}\rightarrow \cdots \rightarrow P_1\xrightarrow{d_1} P_0\xrightarrow{d_0}  M \rightarrow  0$$
		be an $n$-projective resolution of $M$. On contrary, assume that $d<m$. Then
		$$\next^1_R(\Coker({d_m}),P_m)\cong \next^m_R (M,P_m)$$
		and the second member is zero by our assumption. Consequently, the sequence
		$$0\rightarrow P_m\xrightarrow{d_m} P_{m-1}\rightarrow \Coker({d_m})\rightarrow 0$$
		splits, which contradicts the assumption $n$-$\pd_RM = m$. And thus $m\leq d$.
	\end{proof}

	Finally, we can characterize the $n$-global dimension of rings in terms of complexes.
	
	\begin{theorem} Let $m\in \mathbb{N}$. Then the following statements are equivalent.
		\begin{enumerate}
			\item $n\mbox{-}\gldim(R)\leq m$.
			\item $n$-$\pd_RX^{\bullet}\leq m-\infn X^{\bullet}$ for any $X^{\bullet}\in \D^b_{n}(R)$.
			\item $n$-$\id_RY^{\bullet}\leq m+\supn Y^{\bullet}$ for any $Y^{\bullet}\in \D^b_{n}(R)$.
			\item $\next^i_R(X^{\bullet},Y^{\bullet}) = 0$ for any $X^{\bullet}, Y^{\bullet}\in \D^b_{n}(R)$ and $i > m + \supn Y^{\bullet}-\infn X^{\bullet}$.
		\end{enumerate}
	\end{theorem}

	\begin{proof}
		$(2)\Rightarrow(4)$ and $(3)\Rightarrow(4)$ follow from Theorem \ref{c-fpn} and Theorem \ref{c-fin}, respectively.  $(4)\Rightarrow(1)$ follows by letting both $X^{\bullet}$ and $Y^{\bullet}$ be $R$-modules. $(1)\Rightarrow(2)$ is the dual of $(1) \Rightarrow(3)$. 
		
		$(1)\Rightarrow (3)$
		Let $\supn Y^{\bullet} = t$ and $Y^{\bullet}\rightarrow I^{\bullet}$ be an $n$-injective coresolution of $Y^{\bullet}$ . Then $\supn I^{\bullet} = t$
		and $I^{\bullet}$ is $n$-exact in degree $\geq t + 1$. So
		$$0 \rightarrow \Ker( d^t_{I^{\bullet}}) \rightarrow I^t\rightarrow I^{t+1} \rightarrow \cdots $$
		is an $n$-injective coresolution of $\Ker (d^t_{I^{\bullet}})$. By $(1)$, we have $n$-$\id_R(\Ker (d^t_{I^{\bullet}}))\leq m$. Let
		$$0 \rightarrow \Ker(d^t_{I^{\bullet}}) \rightarrow K^0 \rightarrow K^1\rightarrow \cdots  \rightarrow K^m \rightarrow 0$$
		be an $n$-injective coresolution of $\Ker( d^t_{I^{\bullet}})$. Then it is easy to check that
		$$\cdots \rightarrow I^{t-2}\rightarrow I^{t-1} \rightarrow K^0 \rightarrow \cdots \rightarrow K^{m-1} \rightarrow K^m \rightarrow 0 \rightarrow \cdots$$
		is an $n$-injective coresolution of $Y^{\bullet}$ and $n$-$\id_RY^{\bullet}\leq m+t$.
	\end{proof}
	\
	\section{$n$-Singularity Categories}
	The singularity category of a given ring $R$, which was introduced by Buchweitz \cite{Buchweitz}, is defined to be the Verdier quotient $\D^b_{sg}(R):= \D^b(R)/\K^b(\Proj)$, where $\K^b(\Proj)$ is the bounded  homotopy category of all projective modules. Since $\D^b_{sg}(R) = 0$ if and only
	if every $R$-module has finite projective dimensions, if and only if $R$ has a finite global dimension. So the singularity category $\D^b_{sg}(R)$ reflects homological singularity of the
	ring $R$.

	In this section, we introduce and study  the notion of  $n$-singularity categories, 
	the relative singularity category with respect to $n$-exact structure. Set 
	\begin{equation*}
		\begin{aligned}
			&\K^{-,\boldsymbol{nb}}(\fP_n):=\{X^{\bullet}\in\K^{-}(\fP_n)\mid \infn X^{\bullet}\ \mbox{is finite} \}\mbox{, and}\\
			&\K^{+,\boldsymbol{nb}}(\fI_n):=\{X^{\bullet}\in\K^{-}(\fI_n)\mid \boldsymbol{sup_n}X^{\bullet}\ \mbox{is finite} \}.
		\end{aligned}
	\end{equation*}
	It follows by \cite[Theorem 5.11]{ZZW} that there is a triangle-equivalence $\D^b_{n}(R)\simeq \K^{-,\boldsymbol{nb}}(\fP_n)$,
	and one can verify as in \cite[Lemma 5.1]{ZZW} that $\K^b(\fP_n)$ is a thick subcategory of $\K^{-,\boldsymbol{nb}}(\fP_n)$. So we can  introduce  $n$-singularity categories as follows. 
	
	\begin{definition} Let $R$ be a ring. The $n$-singularity category of $R$ is defined to be the Verdier quotient
		$$\D^b_{n\mbox{-}sg}(R) := \D^b_{n}(R)/\K^b(\fP_n)\simeq \K^{-,\boldsymbol{nb}}(\fP_n)/\K^b(\fP_n).$$
	\end{definition}
	
	Since the $0$-pure exact structure is exactly the  exact structure, the $0$-pure singularity category is precisely the classical  singularity category. 
	 The following result exhibits when the $n$-singularity category  $\D^b_{n\mbox{-}sg}(R)$ of a given ring $R$ vanishes.
	
	\begin{theorem}\label{finite=0} Let $R$ be a ring. Then $\D^b_{n\mbox{-}sg}(R) = 0$ if and only if $n\mbox{-}\gldim(R)$ is finite. Moreover, if $n\in\mathbb{N}$, then the above are also equivalent to  $\gldim(R)$ is finite $($which is also equivalent to $\D^b_{sg}(R) = 0).$
	\end{theorem}
	\begin{proof} Suppose the $n$-global dimension $n\mbox{-}\gldim(R)$ of $R$ is finite.  Then it follows from Theorem \ref{c-fpn} that for any $X^{\bullet} \in \D^b_{n}(R)$, there exists an $n$-quasi-isomorphism $P^{\bullet}\rightarrow X^{\bullet}$ with
		$P^{\bullet} \in \K^b(\fP_n)$. So $\D^b_{n}(R)\simeq \K^b(\fP_n)$, and hence $\D^b_{n\mbox{-}sg}(R) = 0$.
		
		On the other hand, suppose $\D^b_{n\mbox{-}sg}(R) = 0$. Let $M$ be an $R$-module. Then
		$M\cong 0$ in $\D^b_{n\mbox{-}sg}(R)$, and so there exists some $X^{\bullet} \in \K^b(\fP_n)$ such that $M \cong X^{\bullet}$ in $\D^b_{n}(R)$.
		We denote this isomorphism by a right fraction
		$$\alpha/f :M\stackrel{f}{\Longleftarrow}Y^{\bullet}\stackrel{\alpha}{\longrightarrow} X^{\bullet},$$
		where $f$ and $\alpha$ are
		$n$-quasi-isomorphisms. Then there is a triangle $$Y^{\bullet} \xrightarrow{\alpha} X^{\bullet}\rightarrow \Cone^{\bullet}(\alpha)\rightarrow Y^{\bullet} [1]$$ in $\K(R)$
		with $\Cone^{\bullet}(\alpha)$ $n$-exact. By applying $\Hom_{\K(R)}(X^{\bullet},-)$ to it, we get an exact sequence
		$$\Hom_{\K(R)}(X^{\bullet},Y^{\bullet})\rightarrow \Hom_{\K(R)}(X^{\bullet},X^{\bullet})\rightarrow \Hom_{\K(R)}(X^{\bullet},\Cone^{\bullet}(\alpha)).$$
		It follows from \cite[Lemma 4.8]{ZZW} that $\Hom_{\K(R)}(X, \Cone^{\bullet}(\alpha)) = 0$, so there is a cochain map
		$\beta : X^{\bullet}\rightarrow Y^{\bullet}$ such that $\alpha\beta$ is homotopic to $\Id_{X^{\bullet}}$ . Thus we have an $n$-quasi-isomorphism
		$f\beta : X^{\bullet}\rightarrow M$.
		Consider the soft truncation $$X_{0\supset}:= \cdots  \rightarrow X^{-2} \rightarrow X^{-1} \rightarrow \Ker( d^0_{X^{\bullet}}) \rightarrow 0$$ of $X^{\bullet}.$ Then there
		is an $n$-quasi-isomorphism $f\beta\lambda : X_{0\supset} \rightarrow M$, where $\lambda : X_{0\supset} \hookrightarrow X^{\bullet}$ is the natural embedding map.
		Since $X^{\bullet}\in \K^b(\fP_n)$, we can  assume that there is an integer $m$ such that $X^i = 0$ for any $i > m$.
		Note that the sequence
		$$0 \rightarrow \Ker (d^0_{X^{\bullet}}) \rightarrow X^0 \rightarrow \cdots  \rightarrow X^{m-1} \rightarrow X^m \rightarrow 0$$
		is $n$-exact with each $X^i \in\fP_n$, and it follows that $\Ker (d^0_{X^{\bullet}})$ is also $n$-projective. Thus
		$M$ has a bounded $n$-projective resolution $X_{0\supset} \rightarrow M$, and so $n\mbox{-}\pd(M) < \infty$. 
		
		Now, let $n\in\mathbb{N}$. Then, by Proposition \ref{pd-fpd}, $n$-\pd$_RM\leq$\pd$_RM\leq n$-\pd$_RM+n$ for any $R$-module $M$.
		Hence the finiteness of $n\mbox{-}\gldim(R)$ is equivalent to that of $\gldim(R)$. 
	\end{proof}
	If take $n=0$ in Theorem \ref{finite=0}, we can recover the following classical result.
	\begin{corollary} Let $R$ be a ring. Then $\D^b_{sg}(R) = 0$ if and only if $\gldim(R)$ is finite.
	\end{corollary}
	
\begin{proposition}\label{n-sg-equ} Assume that $\fP_n$ is closed under kernels of epimorphisms. Then $\D^b_{n\mbox{-}sg}(R)$  is triangulated equivalent to $\D^b_{sg}(R)$ for any $n\in\mathbb{N}$.
\end{proposition}	
\begin{proof} It follows by \cite[Theorem 6.6]{ZZW} that $\D^b(R)\simeq \K^{-,\boldsymbol{nb}}(\fP_n)/\K_{ac}^b(\fP_n)$. Note that $\D^b_{sg}(R)\simeq \D^b(R)/\K^b(\fP_0)$ and $\D^b_{n\mbox{-}sg}(R) \simeq \K^{-,\boldsymbol{nb}}(\fP_n)/\K^b(\fP_n).$ It follows by Proposition \ref{pd-fpd} that  $\K^b(\fP_n)\simeq  \K^b(\fP_0)$ for any $n\in\mathbb{N}$.  Hence, we have $\D^b_{n\mbox{-}sg}(R)\simeq\D^b_{sg}(R)$ for any $n\in\mathbb{N}$.
\end{proof}	
\begin{remark}
	We do not know that whether the condition ``Assume that $\fP_n$ is closed under kernels of epimorphisms.'' in Proposition \ref{n-sg-equ} can be omitted.
\end{remark}

Recall the notion of recollement of triangulated categories, which is introduced by Beilinson, Bernstein
and Deligne \cite{Beilinson}. Let $\mathcal{T}', \mathcal{T}$ and $\mathcal{T}''$ be three triangulated  categories. We say that $\mathcal{T}$ admits a recollement relative to $\mathcal{T}'$ and $\mathcal{T}''$ , if there exist six triangulated functors as in the following diagram
$$\xymatrix{\mathcal{T}'\ar[rr]!R|{i_{\ast}}&&\ar@<-2ex>[ll]!R|{i^{\ast}}\ar@<2ex>[ll]!R|{i^{!}}\mathcal{T}
		\ar[rr]!L|{j^{\ast}}&&\ar@<-2ex>[ll]!L|{j_{!}}\ar@<2ex>[ll]!L|{j_{\ast}}\mathcal{T}''}$$
	satisfying the following items
	\begin{enumerate}
		\item $(i^{\ast},i_{\ast}), (i_{\ast}, i^!), (j_{!},j^{\ast})$ and $(j^{\ast},j_{\ast})$ are adjoint pairs;
		
		\item $i_{\ast}, j_{\ast}$ and $j_{!}$ are full embedding;
		
		\item $j^{\ast}i_{\ast} = 0$;
		
		\item for each $X\in \mathcal{T}$, there are distinguished triangles
		$$i_{\ast}i^!(X) \rightarrow X \rightarrow j_{\ast}j^{\ast}(X) \rightarrow i_{\ast}i^!(X)[1],$$
		$$j_!j^{\ast}(X) \rightarrow X \rightarrow i_{\ast}i^{\ast}(X) \rightarrow j_!j^{\ast}(X)[1].$$
	\end{enumerate}
	
	\begin{theorem}\label{main} Let $R, S$ and $T$ be rings. Assume that $\D^b_{n}(R)$ admits the following recollement
		$$\xymatrix{\D^b_{n}(S)\ar[rr]!R|{i_{\ast}}&&\ar@<-2ex>[ll]!R|{i^{\ast}}\ar@<2ex>[ll]!R|{i^{!}}\D^b_{n}(R)\ar[rr]!L|{j^{\ast}}&&\ar@<-2ex>[ll]!L|{j_{!}}\ar@<2ex>[ll]!L|{j_{\ast}}\D^b_{n}(T)}$$\\
		Then $\D^b_{n\mbox{-}sg}(R) = 0$ if and only if $\D^b_{n\mbox{-}sg}(S) = 0 = \D^b_{n\mbox{-}sg}(T)$. Consequently, $n\mbox{-}\gldim(R)$ is finite if and only if so are $n\mbox{-}\gldim(S)$ and $n\mbox{-}\gldim(T)$.
	\end{theorem}
	\begin{proof} Since $\D^b_{n}(S)$ and $\D^b_{n}(T)$ can be fully embedded into 
		$\D^b_{n}(R)$ by the functors $i_{\ast}$ and $j^{\ast}$ respectively, we have
		the finiteness of $n\mbox{-}\gldim(R)$ implies the finiteness of both $n\mbox{-}\gldim(S)$ and $n\mbox{-}\gldim(T)$.  It follows by
		Theorem \ref{finite=0} that $\D^b_{n\mbox{-}sg}(R) = 0$ implies $\D^b_{n\mbox{-}sg}(S) = 0 =\D^b_{n\mbox{-}sg}(T)$.
		
		On the other hand, suppose $\D^b_{n\mbox{-}sg}(S) = 0 =\D^b_{n\mbox{-}sg}(T)$. Then $S$ and $T$ have finite $n$-global dimensions. Let $M$ and $N$ be any $R$-modules. The above recollement induces the
		following two distinguished triangles in $\D^b_{n}(R)$:
		$$j_!j^{\ast}(M) \rightarrow M \rightarrow i_{\ast}i^{\ast}(M) \rightarrow j_!j^{\ast}(M)[1],$$
		$$i_{\ast}i^!(N) \rightarrow N \rightarrow j_{\ast}j^{\ast}(N) \rightarrow i_{\ast}i^!(N)[1].$$
		By applying $\Hom_{\D^b_{n}(R)}(-,i_{\ast}i^!(N))$ and $\Hom_{\D^b_{n}(R)}(-, j_{\ast}j^{\ast}(N))$ to the first triangle, we
		get the following long exact sequences  for every $m \in \mathbb{Z}$:
		$$ \Hom_{\D^b_{n}(R)}(i_{\ast}i^{\ast},i_{\ast}i^!(N)[m]) \rightarrow \Hom_{\D^b_{n}(R)}(M,i_{\ast}i^!(N)[m]) \rightarrow \Hom_{\D^b_{n}(R)}(j_!j^{\ast},i_{\ast}i^!(N)[m]),$$
		$$\Hom_{\D^b_{n}(R)}(i_{\ast}i^{\ast}, j_{\ast}j^{\ast}(N)[m]) \rightarrow \Hom_{\D^b_{n}(R)}(M, j_{\ast}j^{\ast}(N)[m])\rightarrow\Hom_{\D^b_{n}(R)}(j_!j^{\ast}, j_{\ast}j^{\ast}(N)[m]).$$
		Since $j^{\ast}i_{\ast} = 0$, we have $$\Hom_{\D^b_{n}(R)}(j_!j^{\ast},i_{\ast}i^!(N)[m])\cong \Hom_{\D^b_{n}(T)}(j^{\ast}(M),j^{\ast}i_{\ast}i^!(N)[m])=0$$ for every $m \in \mathbb{Z}$; and similarly, $\Hom_{\D^b_{n}(R)}(i_{\ast}i^{\ast},j_{\ast}j^{\ast}(N)[m])=0$ follows by
		the adjunction $(j^{\ast}, j_{\ast})$. Since $n\mbox{-}\gldim(S) <\infty$, we have $\D^b_{n}(S)\simeq \K^b(\fP_n)$. Noting that $i^{\ast}(M)$ and $i^!(N)$ lie in $\K^b(\fP_n)$, one has $\Hom_{\D^b_{n}(R)}(i^{\ast}(M), i^{!}(N)[m])=0$ for $m\gg 0$; moreover, $i_{\ast}$ is a full embedding map and then $\Hom_{\D^b_{n}(R)}(i_{\ast}i^{\ast}(M),i_{\ast}i^!(N)[m])=0$ for $m\gg 0$. Similarly,   $\Hom_{\D^b_{n}(R)}(j_!i_{\ast}(M), j_{\ast}j^{\ast}(N)[m])=0$ for $m\gg 0$ since
		$n\mbox{-}\gldim(T) < \infty$.
		
		Now we apply $\Hom_{\D^b_{n}(R)}(M, -)$ to the second triangle, and obtain the following long
		exact sequence
		$$\Hom_{\D^b_{n}(R)}(M,i_{\ast}i^!(N)[m]) \rightarrow \Hom_{\D^b_{n}(R)}(M,N[m])\rightarrow \Hom_{\D^b_{n}(R)}(M,j_{\ast}j^{\ast}(N)[m]).$$
		By the above argument, we have $$\Hom_{\D^b_{n}(R)}(M,i_{\ast}i^!(N)[m]) = 0 = \Hom_{\D^b_{n}(R)}(M,j_{\ast}j^{\ast}(N)[m])$$ for $m\gg 0$. Then $\next^n_R(M, N)\cong \Hom_{\D^b_{n}(R)}(M,N[m]) = 0$ by Proposition \ref{n-ext}. This implies
		that $n\mbox{-}\gldim(R) < \infty$. And hence  $\D^b_{n\mbox{-}sg}(R) = 0$  by Theorem \ref{finite=0}. 
	\end{proof}
	
	If we take $n=0$ in Theorem \ref{main}, we can  get the following  result.
	\begin{corollary}  Let $R, S$ and $T$ be rings. Assume that $\D^b(R)$ admits the following recollement
		$$\xymatrix{\D^b(S)\ar[rr]!R|{i_{\ast}}&&\ar@<-2ex>[ll]!R|{i^{\ast}}\ar@<2ex>[ll]!R|{i^{!}}\D^b(R)\ar[rr]!L|{j^{\ast}}&&\ar@<-2ex>[ll]!L|{j_{!}}\ar@<2ex>[ll]!L|{j_{\ast}}\D^b(T)}$$\\
		Then $\D^b_{sg}(R) = 0$ if and only if $\D^b_{sg}(S) = 0 = \D^b_{sg}(T)$.  Consequently, $\gldim(R)$ is finite if and only if so are $\gldim(S)$ and $\gldim(T)$.
	\end{corollary}

\begin{acknowledgement}
The third author was supported by the NSF of China (No. 12271292).
\end{acknowledgement}

\end{document}